\documentclass[12pt,twoside]{article}
\usepackage[dvips]{graphicx}
\usepackage{latexsym}
\usepackage{amsthm}
 
\usepackage{amssymb}

\def\R{{\mathbb R}}

\newtheorem*{theorem*}{Theorem}
\newtheorem{theorem}{Theorem}   

\newtheorem{prop}[theorem]{Proposition} 
\newtheorem{lemma}[theorem]{Lemma}

\newtheorem{definition}[theorem]{Definition}

\def\R{{\mathbb R}} 

\begin{document}
\overfullrule=0pt
\baselineskip=24pt
\font\tfont= cmbx10 scaled \magstep3
\font\sfont= cmbx10 scaled \magstep2
\font\afont= cmcsc10 scaled \magstep2
\title{\tfont 
On Gauss-Kuzmin Statistics and the Transfer Operator for a Multidimensional Continued Fraction Algorithm: the Triangle Map}
\bigskip
\author{Thomas Garrity\\  
Department of Mathematics\\ Williams College\\ Williamstown, MA  01267\\ 
email:tgarrity@williams.edu }

\date{}
\maketitle
\begin{abstract}
The Gauss-Kuzmin statistics for the triangle map (a type of multidimensional continued fraction algorithm) are derived by examining the leading eigenfunction of the triangle map's transfer operator.  The technical difficulty is finding the appropriate Banach space of functions.  We also show that, by thinking of the triangle map's transfer operator as acting on a one-dimensional family of Hilbert spaces, the transfer  can be thought of as a family of nuclear operators of trace class zero.

\end{abstract}

\section{Introduction}
It is classical to think of continued fractions as iterations of the Gauss map on  the unit interval and then  to consider the underlying dynamical system.  For example,  the Gauss-Kuzmin statistics for continued fractions can be reduced to the study of the spectrum of the associated transfer operator \cite{ Hensley, Iosifescu-Kraaikamp1, Khinchin, Rockett-Szusz, Schweiger3}.  Most multidimensional continued fraction algorithms are iterations of a triangle $\triangle$ (though  the Jacobi-Perron algorithm acts on a square). For background on many types of multidimensional continued fraction algorithms, see Schweiger's {Multidimensional Continued Fractions} \cite{Schweiger4} and Karpenkov's {\it Geometry of Continued Fractions} \cite{Karpenkov13}.  Thus it is natural and standard  to think of these as dynamical systems.  This paper concentrates on the transfer operator for the triangle map \cite{GarrityT05,Garrity1, Karpenkov13, Messaoudi-Nogueira-Schweiger09, Schweiger05} and then uses results on the spectrum of this operator to look at the triangle map's Gauss-Kuzmin statistics.

We will show that the transfer operator for the triangle map is 
$$\mathcal{L}(f)(x,y) = \sum_{n=0}^{\infty} \frac{1}{(1+nx+y)^3} f \left( \frac{1}{1+kx+y},  \frac{x}{1+kx+y} \right)$$
Note that though this transfer operator looks similar to the transfer operator for the Gauss map: $$\sum  \frac{1}{(n+x)^2}f\left(\frac{1}{n+x}  \right),$$  there is one significant difference, namely that in the triangle transfer operator there is the  term  $nx$ in the denominator, 
while the continued fraction transfer operator has the term $n+x$ in the denominator.  This difference is what prevents us from applying the standard  bounds for the continued fraction case to the triangle case.  It is this difference that to a large extent makes this paper not simply an easy generalization of earlier work.

In this paper we find an appropriate Banach space of functions for which we can show that $1$ is the largest eigenvalue, with one-dimensional eigenspace,  of the transfer operator.  This can then be used to understand the Gauss-Kuzmin statistics for the triangle map.  We then show, in analog to the Gauss map, that there is also a Hilbert space approach.  We will see that in an appropriate sense the triangle map's transfer operator can be thought of as a nuclear operator.  Unlike the continued fraction case,  we cannot be working  in a Hilbert space of  square-integrable functions on the  two-dimensional domain $\triangle,$ as the natural eigenfunction with eigenvalue one is not square-integrable.  Instead we will be thinking of the transfer operator as acting on a space of functions $f(x,y)$ that are only  square-integrable with respect to the variable $y$.  Thus we will be acting on a family of Hilbert spaces parameterized by the variable $x$.

There are two overall goals for this paper.  First is to  begin the spectral analysis for the transfer operator of the triangle map. Second, though, is to set-up the machinery for joint work with Ilya Amburg  \cite{Amburg-Garrity15}.  While the triangle map, to some extent, is just one example among many of possible multidimensional continued fraction algorithms, in \cite{SMALL11q1, SMALL11q3}  it is shown that, by varying the triangle map in a natural way,  a whole collection of both new and old multidimensional continued fraction algorithms can be generated, called triangle partition maps.  Further, in that work, it is shown that these triangle partition maps generate a family of multidimensional continued fraction algorithms, called combination triangle partition maps,  that in turn  are shown  to include  many, if not most, known multidimensional continued fractions.  In \cite{Amburg-Garrity15}, it will be shown that the transfer operators for half of the triangle partition maps have behavior similar to that of the triangle map.  This will allow us to apply the results of this paper to this other work.  More interestingly, we will see that the transfer operators of the other   half of the triangle partition maps have  remarkably different properties.  

We would like to thank Ilya Amburg for  a lot of help on this paper, including collaboration on the actual formulas for the Gauss-Kuzmin statistics.  We would also like to thank L. Pedersen for many useful comments.

\section{Transfer Operator for Triangle Map}

 Subdivide $\triangle =\{(x,y)\in \R^2: 1>x>y>0\}$ into subtriangles
 $$\triangle_k = \{(x,y)\in \triangle: 1-x- ky\geq 0> 1-x -(k+1)y\}$$  
   \begin{definition} The triangle map $T:\triangle \rightarrow \triangle$ is defined by setting $T=T_k$ for $(x,y) \in \triangle_k,$ and in turn setting
 $$T_k(x,y)= \left( \frac{y}{x}, \frac{1-x-ky}{x} \right).$$
The point $(x,y)\in \triangle$ has triangle sequence $(a_0,a_1, a_2, \ldots )$ if 
 $$T^N(x,y) \in \triangle_{a_N}.$$
 \end{definition}
 For more on the triangle map, see \cite{GarrityT05,Garrity1, Karpenkov13, Messaoudi-Nogueira-Schweiger09, Schweiger05}

 To define the triangle map's transfer operator, we first compute the Jacobian:
 $$J(x,y)  =  \det  \left( \begin{array}{cc}  \frac{\partial}{\partial x} \left(\frac{y}{x}\right) & \frac{\partial}{\partial y} \left(\frac{y}{x}\right) \\ &\\ \frac{\partial}{\partial x} \left(\frac{1-x-ky}{x}\right)&  \frac{\partial}{\partial y} \left(\frac{1-x-ky}{x}\right) \end{array} \right) = \det  \left( \begin{array}{cc}  -\frac{y}{x^2}    & \frac{1}{x}   \\           
  \frac{ky-1}{x^2}       &  \frac{-k}{x} \end{array}   \right) = \frac{1}{x^3},$$
for $(x,y)\in \triangle_k.$

Next we need to find the inverses $t_k:\triangle_k \rightarrow \triangle $ of the triangle map, one for each non-negative integer $k$.  These are 
$$t_k(x,y) = \ \left( \frac{1}{1+kx+y},  \frac{x}{1+kx+y} \right)$$
which can be checked by  direct calculation,.

By definition, for any differentiable map $T$, the transfer operator is 
$$\mathcal{L}_T(F)(p) = \sum_{q:T(q)=p} \frac{1}{\mbox{Jac}(T(q))} f(q).$$ This leads to 
\begin{prop} The transfer operator for the triangle map is
$$\mathcal{L}(f)(x,y) = \sum_{n=0}^{\infty} \frac{1}{(1+nx+y)^3} f \left( \frac{1}{1+kx+y},  \frac{x}{1+kx+y} \right).$$
\end{prop}

This map is also be called  the Perron-Frobenius operator or the Ruelle-Perron-Frobenius operator with respect to the Lebesgue measure $\mathrm{d}x\mathrm{d}y,$ having the property that  for all $f\in L^1(\mathrm{d}x\mathrm{d}y)$, 
and all measurable subsets $A \subset \triangle$, 
$\int_A \mathcal{L}_T(f) \mathrm{d}x\mathrm{d}y = \int_{T^{-1}A} f  \mathrm{d}x\mathrm{d}y.$
For more on transfer maps in general, see Baladi  \cite{Baladi00} and for continued fractions, see Iosifescu and Kraaikamp \cite{Iosifescu-Kraaikamp1}.

By direct calculation,  for  the function 
$$f(x,y) = \frac{1}{x(1+y)}$$
we have
that
$$\mathcal{L}(f)(x,y) = f(x,y).$$
  What we now must do is find appropriate vector spaces of functions on which the operator $\mathcal{L}$ acts so that $f(x,y)$ is the leading  eigenfunction, meaning that we would know enough about the spectrum of $\mathcal{L}$ to be able to state that the largest eigenvalue is one, and that it has a one-dimensional eigenspace.

\section{Banach Space Approach}
We want to identify a Banach space of real-valued functions on $\triangle$ on which the transfer operator has leading eigenfunction $ \frac{1}{x(1+y)}$ with eigenvalue one.

Let $\mathcal{C}(\triangle)$ be the vector space of real-valued  continuous functions on the triangle $\triangle$.  Note that since $\triangle $ is open, the elements of $\mathcal{C}(\triangle)$ need not be bounded.

Set 
$$V=\{f \in \mathcal{C}(\triangle): \exists C\in \R \;\mbox{such that} \; |xf(x,y)| < C, \forall (x,y)\in \triangle \}.$$
This is a Banach space under the norm
$$||f||= \sup_{(x,y\in \triangle} |xf(x,y)| $$

\begin{theorem} The transfer operator is a continuous linear map from $V$ to itself.

\end{theorem}

\begin{proof}  Linearity is immediate.  We will show for all $f\in V$ that
$$|| \mathcal{L} f(x,y)| | \leq 3 ||f||.$$

For any  $f\in V,$  let $C = ||f|| $, which means that for all $(x,y) \in \triangle$ we have
$$|xf(x,y)| < C.$$
Then
\begin{eqnarray*}
|x   \mathcal{L} f(x,y)| &=& \left| \sum_{n=0}^{\infty}\left( \frac{x}{(1+nx+y)^{3}}  \right) f\left(  \frac{1}{1+nx+y}, \frac{x}{1+nx+y}  \right)  \right|  \\
&\leq&  \sum_{n=0}^{\infty}  \left|  \left( \frac{x}{(1+nx+y)^{3}}  \right) f\left(  \frac{1}{1+nx+y}, \frac{x}{1+nx+y}  \right) \right|  \\
&=& \sum_{n=0}^{\infty}  \left|  \left( \frac{x}{(1+nx+y)^{2}}  \right)\left( \frac{1}{1+nx+y}\right) f\left(  \frac{1}{1+nx+y}, \frac{x}{1+nx+y}  \right) \right|  \\
&\leq&  \sum_{n=0}^{\infty}\left( \frac{Cx}{(1+nx+y)^{2}}  \right) \\
&=& Cx  \sum_{n=0}^{\infty}\left( \frac{1}{(1+nx+y)^{2}}  \right) \\
&\leq& Cx \left(    \int_0^{\infty} \left( \frac{1}{((1+y)+xt)^{2}}  \right) \mathrm{d}t + \frac{1}{(1+y)^{2}}   \right) \\
&\leq& Cx \left(   \frac{1}{x(1+y)} + 1 \right) \\
&\leq& 3C,
\end{eqnarray*}
giving us that the transform map is indeed a continuous map of the Banach space $V$ to itself.

\end{proof}

Before proving that the largest eigenvalue of  $\mathcal{L}$ is one, with multiplicity one, we need a few lemmas.

We have
\begin{lemma} For two functions $f,g\in V,$ suppose 
 $f(x,y)\leq g(x,y).$ Then for 
 $$ \mathcal{L}(f)(x,y) \leq \mathcal{L}(g)(x,y)$$
 for all $(x,y)\in \triangle.$
\end{lemma}

\begin{proof} This follows from looking at 
$$0 \leq g(x,y) - f(x,y)$$ and then showing that 
$$0 =  \mathcal{L}(0) \leq  \mathcal{L}(g - f))(x,y)  =  \mathcal{L}(g)(x,y) - \mathcal{L}(f)(x,y),$$
which is clear.

\end{proof}





We will also need

\begin{lemma} For  any  $f\in V,$ there is a positive constant $B$ such that for all $(x,y)\in \triangle$,
we have 
$$    \frac{-B}{x(1+y)} \leq f(x,y) \leq  \frac{B}{x(1+y)}.$$

\end{lemma}

\begin{proof} Let $f \in V$ be any element in $V$.  Then there is a $B$ so that 
$$|xf| < \frac{B}{2}.$$
(The reason for the $1/2$ will be clear in a moment.)
Then we have 
$$|f| < \frac{B}{2x} < \frac{B}{x(1+y)},$$
as desired.

\end{proof}

\begin{theorem} The largest eigenvalue of $\mathcal{L}:V\rightarrow V$ is one, with multiplicity one.
\end{theorem}

\begin{proof}   

(The overall structure of this proof is similar to the corresponding result for the Gauss map.  Again, what is new is finding the correct Banach space and using the fairly recent result of Messaoudi, Nogueira, and Schweiger \cite{Messaoudi-Nogueira-Schweiger09} that the triangle map is ergodic.)

Let $f\in V$ be an eigenfunction of $ \mathcal{L}$ with eigenvalue $\lambda.$

We know there is a constant $B$ such that   for all $(x,y) \in \triangle $ 
$$\frac{-B}{x(1+y)} \leq f(x,y) \leq \frac{B}{x(1+y)}.$$
Then for all positive $n$,
$$\frac{-B}{x(1+y)} \leq \mathcal{L}^{(n)}f(x,y) \leq \frac{B}{x(1+y)}$$
and hence
$$\frac{-B}{x(1+y)} \leq \lambda^nf(x,y) \leq \frac{B}{x(1+y)}.$$
Thus we need 
$$|\lambda | \leq 1.$$

We now have to show that this eigenvalue has multiplicity one.
To prove this from scratch would be hard, but using that the triangle map has been proven to be ergodic by Messaoudi, Nogueira, and Schweiger  \cite{Messaoudi-Nogueira-Schweiger09}, the result follows immediately from  Theorem 4.2.2 in Lasota and  Mackey's {\it Probabilistic Properties of Deterministic Systems} \cite{Lasota-Mackey85}  which states 
\begin{theorem*}Let $(X, \mathcal{A}, \mu)$ be a measure space, $S:X\rightarrow X$ a non-singular transformation, and $P$ the Frobenius-Perron operator associated to $S$.  If $S$ is ergodic, then there is at most  one stationary density $f_*$ of $P$.  Further, if there is a unique stationary density $f_*$ of $P$ and $f_*(x) >0$, a.e., then $S$ is ergodic.
\end{theorem*}

Here $P$ is our $\mathcal{L}$.  This gives us our result.

\end{proof} 

It would be interesting to give an argument for this along the lines of the first few chapters in \cite{krsanoselskiii64}.

\section{The Gauss-Kuzmin Distribution for Triangle Maps}
(This section is joint work with Ilya Amburg, and is a special case of chapter 16 in \cite{Amburg2014}).

We would like to know the statistics behind a number's triangle sequence.   Knowing now that the transfer operator has leading eigenvector $1/x(1+y)$, coupled with the work of Messoundi, Noguira and Schweiger \cite{Messaoudi-Nogueira-Schweiger09} showing that the triangle map is ergodic, allows us to apply standard theorems to explicitly determine these statistics.

Let $(x,y) \in \triangle $ have triangle sequence $(a_1,a_2, a_3, \ldots ).$
We define
$$P_{n,k} (x,y) = \frac{\#  \{a_i:a_i=k\; \mbox{and}\; 1\leq i\leq n\}  }{n} ,$$
or, in other words, the percentage of the $a_i$ that are equal to $k$ in the first $n$  terms of the triangle sequence.
Provided the limit exists, we set 
$$P_k(x,y) = \lim_{n\rightarrow \infty} p_{n,k} (x,y) .$$

We know that $\mathrm{d} \mu =\frac{12}{\pi^2x(1+y)}\mathrm{d}x \mathrm{d} y$ is an invariant measure with respect to the triangle map $T:\triangle \rightarrow \triangle$.  (The extra $12/\pi^2$ is just to make  the measure of the domain $\triangle$  be one.)  Then in direct analog to Theorem 9.14 in \cite{Karpenkov13}, we have

\begin{theorem} For almost all $(\alpha, \beta) \in \triangle$,
$$P_k(\alpha, \beta) = \int_{\triangle_k} \mathrm{d} \mu =  \int_{\triangle_k} \frac{12}{\pi^2x(1+y)}\mathrm{d}x \mathrm{d} y.$$
\end{theorem}

The proof is exactly analogous to that in \cite{Karpenkov13}. 

\begin{theorem} For the triangle map $T$, we have that 
$$P(0) =1 -  \frac{6\mathrm{Li}_2\left( \frac{1}{4} \right) + 12\log^2(2) }{\pi^2}$$
and, for $k>0$,
\begin{eqnarray*}P(k)&=& \frac{6}{\pi^2} \left[ \mathrm{Li}_2\left( \frac{1}{(k_1)^2} \right)\right.   - \mathrm{Li}_2\left( \frac{1}{(k+2)^2} \right) \\
&&  + 4\log^2(k+1)- 2\log^2\left( \frac{k+2}{k+1} \right) -2\log(k(k+2))\log(k+1) \bigg]  \end{eqnarray*}
\end{theorem}
Here $ \mathrm{Li}_2$ is the dilogarithm function.

\begin{proof}
These are calculations, for which we used Mathematica.
We have from the previous theorem  that $P_k(\alpha, \beta)$ is independent  of $(\alpha, \beta) \in \triangle$ for almost all elements in the domain. Then
we have 
\begin{eqnarray*}
P(0) &=& \int_{\triangle_k} \frac{12}{\pi^2x(1+y)}\mathrm{d}x \mathrm{d} y  \\
&=& \int_{\frac{1}{2}}^1\left( \int_{1-x}^x   \frac{12}{\pi^2x(1+y)} \mathrm{d}y\right)\mathrm{d}x \\
&=& 1 -  \frac{6\mathrm{Li}_2\left( \frac{1}{4} \right) +12\log^2(2) }{\pi^2}
\end{eqnarray*}
and, for $k>0$,
\begin{eqnarray*}
P(k) &=& \int_{\triangle_k} \frac{12}{\pi^2x(1+y)}\mathrm{d}x \mathrm{d} y  \\
&=& \int_{\frac{1}{k+1}}^1\left( \int_{\frac{1-x}{k+1}}^{\frac{1-x}{k}}   \frac{12}{\pi^2x(1+y)} \mathrm{d}y\right)\mathrm{d}x +  \int_{\frac{1}{k+2}}^{\frac{1}{k+1}}\left( \int_{\frac{1-x}{k+1}}^{x}   \frac{12}{\pi^2x(1+y)} \mathrm{d}y\right)\mathrm{d}x   \\
&=&\frac{6}{\pi^2} \left[ \mathrm{Li}_2\left( \frac{1}{(k+1)^2} \right)\right.   - \mathrm{Li}_2\left( \frac{1}{(k+2)^2} \right) \\
&&  + 4\log^2(k+1)- 2\log^2\left( \frac{k+2}{k+1} \right) -2\log(k(k+2))\log(k+1) \bigg]  \end{eqnarray*}

\end{proof}

\section{Attempts at nuclearity/ On various Hilbert spaces of functions}

For the Gauss map, Mayer and Roepstroff \cite{Mayer-Roepstorff1987} showed that the transfer operator is nuclear of trace class zero.  This section shows that nontrivial analogs hold for the triangle map.  The difficulty in part stems from that the function $1/x(1+y)$ is not in the Hilbert space $L^2(\mathrm{d}x\mathrm{d} y)$, in contrast to the leading eigenfunction $1/(1+x)$ of the Gauss map's transfer operator  being in $L^2(\mathrm{d} x)$.  Instead, we will think of $1/x(1+y)$ as in $L^2(\mathrm{d} y)$, treating the variable $x$ as a parameter.

On the positive reals, set
$$\mathrm{d}m(t) = \frac{t}{e^t-1} \mathrm{d}t.$$
Then set
\begin{eqnarray*}
\eta_k(s) &=& \frac{s^ke^{-s}}{(k+1)!} \\
e_k(t) &=&L_k^1(t) \\
E_k(x,y) &=&   \int_0^{\infty}   \left( \frac{1}{x^2}\right)    e^{-t\left(\frac{1-x+y}{x}\right)}  e_k(t) \mathrm{d}m( t )\\
\mathcal{K}_{T}(\phi(x,t)) &=& \int_0^{\infty} \frac{J_1(2\sqrt{st})}{\sqrt{st}} \frac{t}{e^t-1} \phi(x,s) \mathrm{d}m(s) \\
\widehat{\phi} \left(xy\right) &=&  \frac{1}{x} \int_{s=0}^{\infty} e^{ -sy}\phi   (x,s) \mathrm{d} m(s) \\
\langle \alpha(s), \beta(s) \rangle &=& \int_0^{\infty} \alpha(s)\beta(s) \mathrm{d} m(s).
\end{eqnarray*}
(Here $L_k^1(t) $ denotes the first Laguerre polynomial and  $J_1$ denotes the Bessel function of order one.)

Our goal is to show that 
\begin{eqnarray*}
\mathcal{L}_{T}f(x,y) &=& \mathcal{L}_{T}\widehat{\phi}(x,y))  \\
&= & \frac{1}{x^2} \int_0^{\infty}  e^{-t\left(\frac{1-x+y}{x}\right)}     \mathcal{K}_{T}(\phi(x,t)) \mathrm{d}m(t)\\
&=& \sum_{k=0}^{\infty}\langle \phi(x,s), \eta_k(s) \rangle E_k(x,y)
\end{eqnarray*}

After making a significant change of variables, we will see that our argument mirrors that given in \cite{Mayer-Roepstorff1987}.

 We start with functions $f(x,y)$ with domain $\triangle$ for which there is a $\phi(x,s)$  that is in  $L^2(\mathrm{d m}(s))$ such that  $f= \widehat{\phi}$ and such that $\mathcal{L}_{T} (f)$ exists.   Note that  for $\phi(x,y) = \frac{1-e^{-y}}{y}$
 $$\frac{1}{x(1+y)} = \widehat{\phi}(x,y)) ;$$
 thus such functions exist and include the function $1/x(1+y).$

We have
\begin{eqnarray*}
\mathcal{L}_{T}(f)(x,y) &=& \sum_{n=0}^{\infty} \frac{1}{(1+nx+y)^3} f\left( \frac{1}{1+nx+y}, \frac{x}{1+nx + y}  \right)  \\
&=&\sum_{n=0}^{\infty} \frac{1}{(1+nx+y)^3} \widehat{\phi }  \left( \frac{1}{1+nx+y}, \frac{x}{1+nx + y} \right) .
\end{eqnarray*}
Set
$$w=\frac{1+y}{x}.$$
It is this change of variables that will allow us to use \cite{Mayer-Roepstorff1987}.

Then we have 
\begin{eqnarray*}
 \frac{1}{(1+nx+y)^3}  &=&  \frac{1}{x^3(n+w)^3} \\
   \phi \left( \frac{1}{1+nx+y}, \frac{x}{1+nx + y}  \right) &=& \phi \left( \frac{1}{x(n+w)}, \frac{1}{n+w}  \right)
\end{eqnarray*}

Then $\mathcal{L}_{T,\mu}(f)(x,y) $ is
$$ \sum_{n=0}^{\infty}  \frac{1}{x^3(n+w)^3}  \widehat{ \phi }\left( \frac{1}{x(n+w)}, \frac{1}{n+w}  \right)$$
which is 
$$ \sum_{n=0}^{\infty}  \frac{1}{x^3(n+w)^3}  (x(n+w)) \int_0^{\infty} e^{ \frac{-s}{n+w} }\phi(x,s) \mathrm{d}m(s)$$
which is 

$$\frac{1}{x^2}  \int_0^{\infty} \sum_{n=0}^{\infty}  \frac{1}{(n+w)^2} e^{ \frac{-s}{n+w} }\phi(x,s) \mathrm{d}m(s).$$

We want to rewrite 
$$ \sum_{n=0}^{\infty}  \frac{1}{(n+w)^2} e^{ \frac{-s}{n+w} }$$
to eliminate the dependence on $n$.

We have 
\begin{eqnarray*}
 \sum_{n=0}^{\infty}  \frac{1}{(n+w)^2} e^{ \frac{-s}{n+w} }  &=&   \sum_{n=0}^{\infty}  \frac{1}{(n+w)^2} \sum_{k=0}^{\infty} \left(\frac{(-s)^k}{k!(n+w)^k}  \right)  \\
 &=&  \sum_{k=0}^{\infty} \left(\frac{(-s)^k}{k!}  \right)  \sum_{n=0}^{\infty} \frac{1}{(n+w)^{k+2}}
 \end{eqnarray*}

From standard arguments on  the Lerch Zeta function, we know that 
\begin{eqnarray*} \sum_{n=0}^{\infty} \frac{1}{(n+w)^{k+2}}   &=& \frac{1}{(k+1)!} \int_{0}^{\infty} \frac{t^{k+1}}{1-e^{-t}} e^{-wt}\mathrm{d}t  \\
&=&\frac{1}{(k+1)!} \int_{0}^{\infty}t^k e^{-(w-1)t}\mathrm{d}m(t) .
\end{eqnarray*}

Thus 
\begin{eqnarray*}
 \sum_{n=0}^{\infty}  \frac{1}{(n+w)^2} e^{ \frac{-s}{n+w} }  &=& \sum_{k=0}^{\infty} \left(\frac{(-s)^k}{k!}  \right) \frac{1}{(k+1)!} \int_{0}^{\infty}t^k e^{-(w-1)t}\mathrm{d}m(t) \\
 &=&\int_{0}^{\infty}   \sum_{k=0}^{\infty} \left(\frac{(-st)^k}{k! (k+1)!}  \right)  e^{-(w-1)t}\mathrm{d}m(t) \\
 &=& \int_{0}^{\infty} \frac{J_1(2\sqrt{st})}{\sqrt{st}}e^{-(w-1)t}    \mathrm{d}m(t),
\end{eqnarray*}
where $J_1$ is the Bessel function of order one.

Then
\begin{eqnarray*}\mathcal{L}_{T,\mu}f(x,y) &=& \frac{1}{x^2}  \int_0^{\infty}       \int_{0}^{\infty} \frac{J_1(2\sqrt{st})}{\sqrt{st}}e^{-(w-1)t}          \phi(x,s)  \mathrm{d}m(t)  \mathrm{d}m(s)  \\
&=& \frac{1}{x^2} \int_0^{\infty}       \int_{0}^{\infty}  \frac{J_1(2\sqrt{st})}{\sqrt{st}}e^{-t\left(\frac{1-x+y}{x}\right)}         \phi(x,s)  \mathrm{d}m(t)  \mathrm{d}m(s)
\end{eqnarray*}

This is why we
set 
$$\mathcal{K}_{T,\mu}(\phi(x,t)) = \int_0^{\infty} \frac{J_1(2\sqrt{st})}{\sqrt{st}}   \phi(x,s) \mathrm{d}m(s) . $$
We have 
\begin{eqnarray*}   \mathcal{L}_{T,\mu}f(x,y)
&=& \frac{1}{x^2}  \int_0^{\infty}       \int_{0}^{\infty}    \frac{J_1(2\sqrt{st})}{\sqrt{st}}   e^{-t\left(\frac{1-x+y}{x}\right)}         \phi(x,s)  \mathrm{d}m(s)  \mathrm{d}m(t)  \\
&=&  \frac{1}{x^2} \int_0^{\infty}    e^{-t\left(\frac{1-x+y}{x}\right)}   \mathcal{K}_{T,\mu}(\phi(x,t))  \mathrm{d}m(t).
\end{eqnarray*}

We want an understanding of the operator $\mathcal{K}_{T,\mu}$.

From \cite{Mayer-Roepstorff1987},  we know that 

$$ \frac{J_1(2\sqrt{st})}{\sqrt{st}}     = \sum_{k=0}^{\infty} L_k^1(t) \frac{s^ke^{-s}}{(n+1)!}$$
where 
$L_k^1(s)$ is a first Laguerre polynomial.

We set 
$$e_k(s) = L_k^1(s), \eta_k(s) = \frac{s^ke^{-s}}{(k+1)!}$$
Both of these are actually functions of only one variable.

Then, also from  \cite{Mayer-Roepstorff1987}, we have

$$\mathcal{K}_{T,\mu}(\phi(x,t)) =   \sum_{k=0}^{\infty}\langle \phi(x,s), \eta_k(s) \rangle e_k(t)  $$
as follows from 
\begin{eqnarray*}
\mathcal{K}_{T,\mu}(\phi(x,t)) &=& \int_0^{\infty} \frac{J_1(2\sqrt{st})}{\sqrt{st}}   \phi(x,s) \mathrm{d}m(s) \\
&=&  \int_0^{\infty}		\sum_{k=0}^{\infty} L_1^k(t) \frac{s^ke^{-s}}{(n+1)!}		\phi(x,s) \mathrm{d}m(s) \\
&=&   \int_0^{\infty}		\sum_{k=0}^{\infty} e_k(t) \eta_k(s) 		\phi(x,s) \mathrm{d}m(s)  \\
&=& \sum_{k=0}^{\infty} e_k(t)    \int_0^{\infty}\eta_k(s) 		\phi(x,s) \mathrm{d}m(s) \\
&=&  \sum_{k=0}^{\infty}\langle \phi(x,s), \eta_k(s) \rangle e_k(t).
\end{eqnarray*}

The functions $\eta_k(s)$ and $e_k(s)$  have already been studied in \cite{Mayer-Roepstorff1987}, where it was shown that 
$$\eta_k(s), e_k(s) \in L^2(\mathrm{d}m(s))$$
and, with respect to this inner product, 
$$\sum_{k=0}^{\infty} |e_k|\cdot |\eta_k| < \infty.$$

Then we have 
\begin{eqnarray*}\mathcal{L}_{T,\mu}f(x,y) &=&  \int_0^{\infty}   \left( \frac{1}{x^2}\right)    e^{-t\left(\frac{1-x+y}{x}\right)}  \left(  \sum_{k=0}^{\infty}\langle \phi(x,s), \eta_k(s) \rangle e_k(t) \right) \mathrm{d}m(t) \\
&=& \sum_{k=0}^{\infty}\langle \phi(x,s), \eta_k(s) \rangle \int_0^{\infty}   \left( \frac{1}{x^2}\right)    e^{-t\left(\frac{1-x+y}{x}\right)}  e_k(t) \mathrm{d}m(t) \\
&=&  \sum_{k=0}^{\infty}\langle \phi(x,s), \eta_k(s) \rangle E_k(x,y),
\end{eqnarray*}
where we set 
$$ E_k(x,y) =  \int_0^{\infty}   \left( \frac{1}{x^2}\right)    e^{-t\left(\frac{1-x+y}{x}\right)}  e_k(t) \mathrm{d}m(t) . $$

We want to show that   $E_k(x,y)\in L^2(\mathrm{d}y)$, again treating the variable $x$ as a parameter.  (The calculation below will also yield that $E_k(x,y)\not\in L^2(  \mathrm{d} x    \mathrm{d}y)$.)

We have 
\begin{eqnarray*}
E_k(x,y)  &=&   \int_0^{\infty} \left( \frac{1}{x^2}\right)  e^{-t\left(\frac{1-x+y}{x}\right)}  L_n^1(t)  \mathrm{d} m_2(t)\\
&=&   \int_0^{\infty} \left( \frac{1}{x^2}\right)  e^{-t\left(\frac{1-x+y}{x}\right)}  L_n^1(t)  \frac{t}{e^t-1} \mathrm{d}t  \\
&=&  \left( \frac{1}{x^2}\right)  \int_0^{\infty} e^{-t\left(\frac{1-x+y}{x}\right)}  L_n^1(t)  \frac{te^{-t}}{1-e^{-t}} \mathrm{d}t  \\
&=&   \left( \frac{1}{x^2}\right)  \int_0^{\infty}  e^{-t\left(\frac{1-x+y}{x}\right)}  L_n^1(t)  te^{-t} \sum_{m=0}^{\infty} e^{-mt}\\
&=&  \left( \frac{1}{x^2}\right)  \int_0^{\infty}     \sum_{m=0}^{\infty}    t L_n^1(t) e^{-s(m)t} \\
&& \mbox{ where $s(m) = \frac{1+y+mx}{x}$}\\
&=&   \left( \frac{1}{x^2}\right)  \sum_{m=0}^{\infty}   \int_0^{\infty}  t L_n^1(t) e^{-s(m)t}  \\
&=&  \left( \frac{1}{x^2}\right)  \sum_{m=0}^{\infty}  \frac{(n+1)(s(m) - 1)^n}{s(m)^{n+2}}   \\
&&\mbox{ by 7.414.8 in \cite{Gradshetyn-Ryzhik80}} \\
&=& (n+1)  \left( \frac{1}{x^2}\right)  \sum_{m=0}^{\infty}   \frac{(      \left( \frac{1+y+mx}{x} \right) - 1)^n}{  \left( \frac{1+y+mx}{x} \right)^{n+2}} \\
&=&  (n+1) \sum_{m=0}^{\infty}  \frac{(1+y+(m-1)x)^n}{(1+y + mx)^{n+2}}
\end{eqnarray*}

As the above series grows like $\sum \frac{1}{m^2x^2}$,  there is a constant $C$ independent of $n$ such that 
$$ (n+1)  \sum_{m=0}^{\infty}  \frac{(1+y+(m-1)x)^n}{(1+y + mx)^{n+2}} < (n+1) C  \left( \frac{1+y}{x^2}\right). $$

\section{Conclusion}
Of course traditional continued fractions have many rich properties.  It is natural to ask which of these properties have analogs for a given multidimensional continued fraction algorithm.
This paper is the start of finding the transfer operator analogs for the triangle map.  In \cite{Amburg-Garrity15},  similar analogs will be developed for triangle partition maps, which as mentioned, while being built in a natural way out of the triangle map, include many if not most multidimensional continued fraction algorithms.  

As mentioned in the introduction, transfer operator have been used for decades in the study of continued fractions, though people did not initially use the framework or rhetoric of functional analysis.  The framework was made explicit in the pioneering work   of Mayer and Roepstroff  \cite{Mayer-Roepstorff1987, Mayer-Roepstorff1988} and of Mayer \cite{Mayer90, Mayer91}, work that continues  even today \cite{Vallee97, Vallee98,Isola02, Antoniou-Shkarin03, Jenkinson-Gonzalez-Urbanksi03, DegliEspost-Isola-Knauf07, Hilgert08, Bonanno-Graffi-Isola08, Alkauskas12, Iosifescu14, Bonanno-Isola14, Ben Ammou-Bonanno-Chouari-Isola15}.   For each of these papers, there are corresponding natural questions for multidimensional continued fractions. The corresponding questions should be non-trivial but interesting.

\end{document}